\documentclass[10pt]{article}

\usepackage{lmodern}
\usepackage[english]{babel}
\usepackage[T1]{fontenc}
\usepackage{babelbib}
\usepackage{cancel}
\usepackage{amsmath}
\usepackage{amssymb}
\usepackage{indentfirst}
\usepackage[latin1]{inputenc}
\usepackage{geometry}
\usepackage{amsfonts}
\usepackage{graphicx}
\usepackage{subfigure}
\usepackage{float}
\usepackage{color}
\usepackage{hyperref}
\hypersetup{colorlinks,linktocpage}
\usepackage{verbatim}

\newcommand{\N}{{\mathbb N}}

\newcommand{\Q}{{\mathbb Q}}

\newcommand{\C}{{\mathbb C}}

\newtheorem{theorem}{Theorem}

\newtheorem{proposition}[theorem]{Proposition}

\newenvironment{proof}{\noindent{\bf Proof.}}{{\hfill $ \Box $}\vskip 4mm}

\definecolor{addresscolor}{rgb}{0,0,0}

\newcommand{\address}[2]{\def\@addressstreet{#1}\def\@addresscity{#2}}
\newcommand{\email}[1]{\def\@email{#1}}

\title{Primitive Central Idempotents of Rational Group Algebras}
\author{Geoffrey Janssens}
\date{} 

\begin{document}

\maketitle

\begin{abstract}
We give a description of the primitive central idempotents of the rational group algebra $\Q G$ of a finite group $G$. Such a description is already investigated by Jespers, Olteanu and del R\'\i o, but some unknown scalars are involved. Our description also gives answers to their questions. 
\end{abstract}


\vspace{0,8cm}

Let $G$ be a finite group. The complex group algebra $\C G$ is semisimple and a description of its primitive central idempotents is well known. These are the elements $e(\chi)= \frac{1}{|G|} \sum_{g\in G} \chi(1) \chi(g^{-1})g$, where $\chi$ runs through the irreducible characters of $G$. Using Galois descent one obtains that the primitive central idempotents of the semisimple rational group algebra $\Q G$ are the elements $e_{\Q}(\chi) = \sum_{\sigma \in G_{\chi}} \sigma(e(\chi))$, with $G_{\chi}= Gal(\Q(\chi)/\Q)$.
Rather recently, Olivieri et al. \cite{ORS} obtained a character free method for describing the primitive central idempotents of $\Q G$ provided that $G$ is a finite monomial group. Their method relies on a theorem of Shoda on pairs of subgroups $(H,K)$ of $G$ with $K$ normal in $H$, $H/K$ abelian an so that an irreducible linear character of $H$ with kernel $K$ induces an irreducible character of $G$. Such pairs are called Shoda pairs. The main ingredient in this theory are the elements $\epsilon(H,K)$ of $\Q H$ with $K \triangleleft H \leq G$. These are defined as $\prod_{M/K \in \mathcal{M}(H/K)}(\tilde{K} - \tilde{M})$ if $H \neq K$ and as $\tilde{H}$ if $H=K$, where $\mathcal{M}(H/K)$ denotes the set of minimal non-trivial normal subgroups of $H/K$ and $\tilde{K}= \frac{1}{|K|} \sum_{k \in K} k$. Furthermore, $e(G,H,K)$ denotes the sum of the different $G$-conjugates of $\epsilon(H,K)$.

For arbitrary finite groups, Jespers, Olteanu and del R\'\i o obtained a description of $e_{\Q}(\chi)$ in \cite{Jespers1}. It is shown that $e_{\Q}(\chi)$ is a $\Q$-linear combination of the elements $e(G, H_{i}, K_{i})$, with $(H_{i}, K_{i})$ Shoda pairs in some subgroups of $G$. They posed the question (\cite[Remark 3.4]{Jespers1}) whether one could determine the scalars and the Shoda pairs involved. 
In this paper we answer both questions by giving a full description of the primitive central idempotents of $\Q G$, for $G$ a finite group.

Throughout G is a finite group. For $\chi$ an arbitrary complex character of $G$ we put:
$e(\chi) = \frac{1}{|G|} \sum_{g\in G} \chi(1) \chi(g^{-1})g$
and
$e_{\Q}(\chi) = \sum_{\sigma \in G_{\chi}} \sigma(e(\chi)).$
%
Note that in general these elements do not have to be idempotents. 
Recall that the M\"obius $\mu$-function, $\mu: \N \rightarrow \{ -1, 0, 1 \}$, is the map defined by $\mu(1)=1$, $\mu(n)= 0$ if $a^{2}|n  \mbox{ with } a > 1$ and $\mu(n)= (-1)^{r}$ if $n= p_{1}p_{2} \ldots p_{r}$ for different primes $p_{1},\ldots, p_{r}$. The induction of a character $\phi$ of a subgroup $H$ to $G$ is defined as $\phi_{H}^{G}(g) = \frac{1}{|H|} \sum_{y \in G} \dot{\phi}(y^{-1} g y)$, where $\dot{\phi}(x)= \phi(x)$ if $x \in H$ and $\dot{\phi}(x)=0$ otherwise. By $1_{G}$ we denote the trivial character of $G$.  


To prove our result we make use of the Artin Induction Theorem. Although this is probably well known, we state and prove it in the following specific form. Recall that for a rational valued character $\chi$ of a group $G$, $\chi(g)= \chi(g^{i})$ for $(i,o(g))=1$.
%
\begin{proposition}(Artin) \label{Artin modified}
If $\psi$ is a rational valued character of $G$, then
$$\psi = \sum_{i=1}^{r} d_{C_{i}} 1_{C_{i}}^{G},$$
where the sum runs through a set $\{ C_{1}, \ldots ,C_{r} \}$ of representatives of conjugacy classes of cyclic subgroups of $G$. Furthermore, if $C_{i} = \langle c_{i} \rangle$ then
$$d_{C_{i}} = \frac{[G:Cen_{G}(c_{i})]}{[G:C_{i}]}\sum_{C_{i}^{*} \geq C_{i}} \mu([C_{i}^{*}:C_{i}]) \psi(z^{*}),$$
where the sum runs through all the cyclic subgroups $C_{i}^{*}$ of $G$ containing $C_{i}$ and $C_{i}^{*} =\langle z^{*} \rangle.$
\end{proposition}
\begin{proof}
For every cyclic subgroup $C=\langle c \rangle$ of $G$, there exists exactly one $i \in \{ 1, \ldots, r \} $ such that $C$ is $G$-conjugated to $C_{i}$. Say, $C= C_{i}^{g^{-1}}$. Set $a_{C}= \frac{|Cen_{G}(c)|}{|G|} d_{C}$.
First we prove that $a_{C} = a_{C_{i}}$ and $1_{C}^{G}= 1_{C_{i}}^{G}$.
To prove the second equality, note that $1_{C}^{G}(g) = \frac{1}{|C|} \sum_{y \in G} \dot{1}_{C}(y^{-1} g y)$, where the function $\dot{1}_{C}(y^{-1} g y)$ is defined as $1$ if $ y^{-1}g y \in C$ and $0$ otherwise. 
This combined with the facts that conjugation preserves the order of subgroups and that it is an automorphism of $G$ we easily see that $1_{C}^{G}= 1_{C_{i}}^{G}$.

Now we prove that $a_{C} = a_{C_{i}}$. Define the sets $(C_{i})\uparrow^{\geq} = \{ K~|~C_{i} \leq K \leq G \}$ and 
$(C)\uparrow^{\geq} = \{ K~|~ C \leq K \leq G \}$. There is a bijective correspondance between these sets. A map from $(C_{i})\uparrow^{\geq}$ to $(C)\uparrow^{\geq}$ is given by conjugation with $g^{-1}$ and the invers map is conjugation by $g$. Along with the fact that $C^{g} = \langle c^{g} \rangle$ if $C = \langle c \rangle$ and $|C|=|C^{g}|$, we see immediatly that $a_{C}= a_{C_{i}}$.

All this yields, $\sum_{C} a_{C} 1_{C}^{G} = \sum_{i=1}^{r} k_{i} a_{C_{i}} 1_{C_{i}}^{G} = \sum_{i=1}^{r} d_{C_{i}} 1_{C_{i}}^{G}$, where $k_{i} = |\mathcal{C}_{G}(c_{i})| = \frac{|G|}{|Cen_{G}(c_{i})|}$ (with $\mathcal{C}_{G}(c_{i})$ the conjugacy class of $c_{i}$ in $G$). The result now follows from Artin's Induction Theorem,\cite[page 489]{Huppert}, which says that every rational valued character of $G$ is of the form $\sum_{C} a_{C} 1_{C}^{G}$, with $a_{C}$ as above and the sum runs over all cyclic subgroups $C$ of $G$.
\end{proof}

Recall that by $\tilde{C_{i}}$ we denote $\epsilon(C_{i},C_{i})= \frac{1}{|C_{i}|} \sum_{c_{i} \in C_{i}} c_{i}$.

\begin{theorem} \label{main theorem}
Let $G$ be a finite group and $\chi$ an irreducible complex character of $G$. Let $C_{i} = \langle c_{i} \rangle$, then we denote
$$b_{C_{i}} = \frac{[G:Cen_{G}(c_{i})]}{[G:C_{i}]} \sum_{C^{*}_{i} \geq C_{i}} \mu([C^{*}_{i}:C_{i}]) (\sum_{\sigma \in G_{\chi}} \sigma(\chi))(z^{*})$$
where the sum runs through all the cyclic subgroups $C_{i}^{*}$ of $G$ which contain $C_{i}$  and $z^{*}$ is a generator of $C_{i}^{*}$.
Then
$$e_{\Q}(\chi) =\sum_{i=1}^{r} \frac{b_{C_{i}} \, \chi(1)}{[G: Cen_{G}(\tilde{C_{i}})]} e(G, C_{i}, C_{i})= \sum_{i=1}^{r} \frac{b_{C_{i}} \, \chi(1)}{[G:C_{i}]}(\sum_{k=1}^{m_{i}} \tilde{C_{i}}^{g_{ik}}),$$
where the first sums runs through a set $\{ C_{1}, \ldots ,C_{r} \}$ of representatives of conjugacy classes of cyclic subgroups of $G$ and $T_{i} = \{ g_{i1},\ldots, g_{im_{i}} \}$ a right transversal of $C_{i}$ in $G$.
\end{theorem}
\begin{proof}
Let $\chi$ be an irreducible complex character of $G$. First we suppose that $\chi(G) \subseteq \Q$. Then $G_{\chi}= \{ 1 \}$ and then by Proposition 1, $\chi = \sum_{i=1}^{r} b_{C_{i}} 1_{C_{i}}^{G}$.

We get
$$\begin{array}{lcl}
e_{\Q}(\chi)= e(\chi) & = & \frac{\chi(1)}{|G|} \sum_{g \in G} (\sum_{i=1}^{r} b_{C_{i}} 1_{C_{i}}^{G}(g^{-1})) g\\
        & = & \frac{\chi(1)}{|G|} \sum_{i=1}^{r} \frac{b_{C_{i}}}{1_{C_{i}}^{G}(1)} \sum_{g \in G} 1_{C_{i}}^{G}(1) 1_{C_{i}}^{G}(g^{-1}) g\\
        & = & \frac{\chi(1)}{|G|} \sum_{i=1}^{r} \frac{b_{C_{i}}}{1_{C_{i}}^{G}(1)} |G| e(1_{C_{i}}^{G})\\
        & = & \sum_{i=1}^{r} \frac{b_{C_{i}} \, \chi(1)}{[G:C_{i}]}e(1_{C_{i}}^{G})\\
\end{array}$$
Let $T_{i}= \{ g_{i1}, \ldots, g_{im_{i}} \}$ be a right transversal of $C_{i}$ in $G$. Then
$$\begin{array}{lcl}
e(1_{C_{i}}^{G}) & = & \frac{1}{|G|} \sum_{g \in G} 1_{C_{i}}^{G}(1) 1_{C_{i}}^{G}(g^{-1})g \\
                 & = & \frac{1}{|G|} \sum_{g \in G} \frac{|G|}{|C_{i}|}1_{C_{i}}(1) 1_{C_{i}}^{G}(g^{-1})g\\
                 & = & \sum_{g \in G} \frac{1}{|C_{i}|} (\frac{1}{|C_{i}|} \sum_{y \in G} \dot{1}_{C_{i}} (yg^{-1} y^{-1}))g\\
                 & = & \sum_{g \in G} \frac{1}{|C_{i}|} (\sum_{j=1}^{m_{i}} \dot{1}_{C_{i}} (g_{ij} g^{-1} g_{ij}^{-1}))g\\
                 & = & \frac{1}{|C_{i}|} \sum_{j=1}^{m_{i}} \sum_{h \in C_{i}} 1_{C_{i}}(h^{-1}) g_{ij}^{-1} h g_{ij}\\
                 & = & \sum_{j=1}^{m_{i}} \tilde{C}_{i}^{g_{ij}}\\
\end{array}$$
With this expression for $e(1_{C_{i}}^{G})$ we obtain one of the equalities in the statement of the result.

Obviously, the sum $\sum_{k=1}^{m_{i}} \tilde{C}_{i}^{g_{ik}}$ adds the elements of the $G$-orbit of $\tilde{C}_{i} = \epsilon (C_{i}, C_{i})$ and each of them $[Cen_{G}(\tilde{C}_{i}) : C_{i}]$ times. So
$$\sum_{k=1}^{m_{i}} \tilde{C}_{i}^{g_{ik}}= [Cen_{G}(\tilde{C}_{i}) : C_{i}] e(G,C_{i}, C_{i}).$$
A simple substitution in the earlier found expression for $e_{\Q}(\chi)$ yields the theorem.

\vspace{0,2cm}
Assume now that $\chi$ is an arbitrary irreducible complex character of $G$.
%
%
%
%
%
Then it is clear and well known that $\sum_{\sigma \in G_{\chi}} \sigma \circ \chi$ is a rational valued character of $G$.
Hence, by the first part we get
$$\begin{array}{lcl}
e(\sum_{\sigma \in G_{\chi}} \sigma(\chi)) & = & \frac{1}{|G|} \sum_{g \in G}(\sum_{\sigma \in G_{\chi}} \sigma(\chi(1)))(\sum_{\sigma \in G_{\chi}} \sigma(\chi(g^{-1})))g \\
                                           & = & \frac{|G_{\chi}| \cdot \chi(1)}{|G|} \sum_{\sigma \in G_{\chi}}\sum_{g \in G} \sigma(\chi(g^{-1})))g\\
                                           & = & \frac{|G_{\chi}| \cdot \chi(1)}{|G|} \sum_{\sigma \in G_{\chi}}\sigma(\sum_{g \in G} \chi(g^{-1})g)\\
                                           & = & |G_{\chi}| e_{\Q}(\chi)\\
\end{array}$$

Since $\sum_{\sigma \in G_{\chi}} \sigma(\chi(1)) = |G_{\chi}| \chi(1)$, the rational case yields the theorem.
\end{proof}
We finish with some remarks. First note that the elements $e(G,C_{i},C_{i})$ are not necessarly idempotens. Second, the definition of $b_{C_{i}}$ is not character-free. However one easily obtains a character free upperbound:
$$b_{C_{i}} \leq \frac{[G:Cen_{G}(c_{i})]}{[G:C_{i}]} \sum_{C_{i}^{*} \geq C_{i}} \mu([C_{i}^{*} : C_{i}]) \phi(n) \chi(1) \leq \frac{[G:Cen_{G}(c_{i})][G:Z(G)]}{[G:C_{i}]} \sum_{C_{i}^{*} \geq C_{i}} \mu([C_{i}^{*} : C_{i}]) \phi(n),$$
where $\phi$ denotes the $\phi$-Euler function. Hence, we can obtain a finite algorithm, that easily can be implemented in for example GAP, to compute all primitive central idempotents of $\Q G$. This answers one of the questions posed in \cite[Remark 3.4]{Jespers1}. Also the description of the idempotents only makes use of pairs of subgroups $(C_{i},C_{i})$, with $C_{i}$ cyclic. This answers the second question posed in \cite[Remark 3.4]{Jespers1}.

\vspace{1,0cm}
G.Janssens \\ 
Departement of Mathematics, Vrije Universiteit Brussel, Pleinlaan 2, 1050, Brussels, Belgium \\
e-mail: geofjans@vub.ac.be



\begin{thebibliography}{99}
\bibitem{Huppert} Bertram Huppert, Character Theory of Finite Groups, De Gruyter expositions in mathematics; 25, 1998
\bibitem{Jespers1} Eric Jespers, Gabriela Olteanu  and Ángel del R\'\i o, Rational  Group Algebras of Finite  Groups:  from Idempotents  to Units of Integral  Group Rings, to appear in, Algebras and Representation Theory
\bibitem{ORS} A.Olivieri, \'A.del R\'\i o, J.J. Sim\'on, On monomial characters and central idempotents of rational group algebras, Communications in Algebra 32 (2004), 1531-1550.
\end{thebibliography}
\end{document}